\documentclass[11pt, twoside]{article}
\usepackage{amssymb, amsmath, amsthm}
\usepackage[left=22mm, right=22mm, bottom=30mm]{geometry}
\usepackage{inputenc}
\usepackage{fancyhdr}
\usepackage{tikz}
\usetikzlibrary{positioning}
\usepackage{titling}
\usepackage{url}
\usepackage{hyperref}
\usepackage{mathtools} 
\predate{}
\postdate{}
\usepackage{lipsum}
\newtheorem{theorem}{Theorem}

\newtheorem{question}[theorem]{Question}
\newtheorem{conjecture}[theorem]{Conjecture}

\begin{document}
\newcommand{\Addresses}{{
\bigskip
\footnotesize
\medskip

\noindent Maria-Romina~Ivan, \textsc{
Department of Pure Mathematics and Mathematical Statistics, Centre for Mathematical Sciences, Wilberforce Road, Cambridge, CB3 0WB, UK.}\\\nopagebreak\textit{Email address: }\texttt{mri25@cam.ac.uk}

\medskip

\noindent Imre~Leader, \textsc{Department of Pure Mathematics and Mathematical Statistics, Centre for Mathematical Sciences, Wilberforce Road, Cambridge, CB3 0WB, UK.}\\\nopagebreak\textit{Email address: }\texttt{i.leader@dpmms.cam.ac.uk}

\medskip

\noindent Mark~Walters, \textsc{School of Mathematical Sciences, Queen Mary University of London, London, E1 4NS, UK.}\\\nopagebreak\textit{Email address: }\texttt{m.walters@qmul.ac.uk}
\medskip}}
\pagestyle{fancy}
\fancyhf{}
\fancyhead [LE, RO] {\thepage}
\fancyhead [CE] {MARIA-ROMINA IVAN, IMRE LEADER AND MARK WALTERS}
\fancyhead [CO] {BLOCK SIZES IN THE BLOCK SETS CONJECTURE}
\renewcommand{\headrulewidth}{0pt}
\renewcommand{\l}{\rule{6em}{1pt}\ }
\title{\Large\textbf{BLOCK SIZES IN THE BLOCK SETS CONJECTURE}}
\author{MARIA-ROMINA IVAN, IMRE LEADER AND MARK WALTERS}
\date{}
\maketitle
\begin{abstract}A set $X$ is called Euclidean Ramsey if, for any $k$ and sufficiently large $n$, every $k$-colouring of $\mathbb{R}^n$ contains a monochromatic congruent copy of $X$. This notion was introduced by Erd\H{o}s, Graham, Montgomery, Rothschild, Spencer and Straus. They asked if a set is Ramsey if and only if it is spherical, meaning that it lies on the surface of a sphere. It is not too difficult to show that if a set is not spherical then it is not Euclidean Ramsey either, but the converse is very much open despite extensive research over the years.\par
On the other hand, the block sets conjecture is a purely combinatorial, Hales-Jewett type of statement, 
concerning `blocks in large products', 
introduced by Leader, Russell and Walters. If true, the block sets conjecture would imply that every
transitive set (a set whose symmetry group acts transitively) is Euclidean Ramsey. As for the question above,  the block sets conjecture remains very elusive, being known only in a few cases.\par
In this paper we show that the sizes of the blocks in the block sets conjecture cannot
be bounded, even for templates over the alphabet of size 3. We also show that for the first non-trivial template, namely $123$, the blocks may be taken to be of size $2$ (for any number of colours). This is best possible; all previous bounds were `tower-type' large.
\end{abstract}
\section{Introduction} 
We start with some background on Euclidean Ramsey theory and the block sets conjecture -- the reader who is already familiar with these may skip this part.

A finite set $X$ in some Euclidean space $\mathbb{R}^d$ is called \textit{Ramsey} if for any $k$ there
exists an $n$ such that whenever $\mathbb{R}^n$ is $k$-coloured there is a monochromatic copy (where 
`copy' means 
congruent copy) of $X$. The study of Ramsey sets was initiated by Erd\H{o}s, Graham, Montgomery, Rothschild, Spencer and Straus \cite{EGMRSS} . They gave examples of Ramsey sets, and proved that a non-spherical set cannot be Ramsey -- here a set is \textit{spherical} if it is a subset
of some sphere. They conjectured that this is the only obstruction to a set being Ramsey, in other
words that $X$ is Ramsey if and only if $X$ is spherical. 

Progress on this question has been very slow. Frankl and R\"odl \cite{FR} showed that every triangle is 
Ramsey, and then that every simplex is Ramsey \cite{FR2}. K\v{r}\'i\v{z} \cite{K} showed that every regular $n$-gon is Ramsey, and 
in fact he showed some more general versions of this statement, which imply for example that every
Platonic solid is Ramsey.

A rival conjecture was formulated by Leader, Russell and Walters in \cite{LRW}. They conjectured that a
set is Ramsey if and only if it is a subset of some finite transitive set (perhaps living in more
dimensions), where a set $X$ is \textit{transitive} if its symmetry group acts transitively on $X$.
It is easy to see that every transitive set must be spherical, but it turns out (although this is
not obvious -- see \cite{LRW} and \cite{LRW2}) that there exist spherical sets that do not embed into any transitive
set. Thus the two conjectures are fundamentally different.

Now, either conjecture would, if true, imply that all transitive
sets are Ramsey. This makes transitive sets a key notion in the area of Euclidean Ramsey theory. The block sets
conjecture, which first appears in \cite{LRW}, is a purely combinatorial statement which implies that every transitive set is
Ramsey. In fact, it is equivalent to a mild strengthening of 
this latter statement: it is equivalent to the assertion that, for
any transitive set $X$, not only is $X$ Ramsey but in fact for any $k$ there is some power $X^n$
such that whenever $X^n$ is $k$-coloured there is a monochromatic
copy of $X$ blown up by a fixed factor $\alpha$ -- in other words, for any $k$ there exist $n$ and $\alpha$ such that $(1/\alpha) X^n$ is `$k$-Ramsey for $X$'.

Before giving a precise definition, let us look at an informal example of a block set. Consider the word $11223$, viewed
as a finite string on the alphabet $[3]=\{ 1,2,3 \}$. Inside the product set $[3]^n$, let us fix five
disjoint sets of coordinates, each of the same size: say $I_1,\dots,I_5$ in $[n]$. We also fix values $a_i$
in $[3]$ for each $i \in [n]$ that does not belong to any $I_j$. And now we form a word $w$ in $[3]^n$ by taking $w_i$ to be $a_i$ for each $i$ not in any $I_j$, and making $w$ constant on each $I_j$, with
these five constant values being $1,1,2,2,3$ in some order. The resulting $30$ points are a `block set with
template $11223$ and  block size $|I_1|$', or simply a `copy' of $11223$.

More formally, we define a \textit{template} over $[m]$ to be a non-decreasing word $T\in[m]^l$ for some $l$. Given a template $T$, we define what we mean by a block set with template $T$. 
Let $S$ be the set of all words on alphabet $[m]$ of length $l$ that are permutations of the
word $T$.  
A \textit{block set} with template $T$ in $[m]^n$ is a set $B$ of words on alphabet $[m]$ of length $n$, formed in the following way. First, select pairwise disjoint subsets $I_i,\dots I_l\subset[n]$ all of the same size $d$ say, and elements $a_i$ for each $i\notin\bigcup_j I_j$. In other words, we select the `block' positions for all letters of $T$, counted with multiplicity, as well as the fixed `reference' word outside them.\par
And now a word $w$ is in $B$ if and only if $w_i=a_i$ for all $i\notin\cup_j I_j$ and there exists $v\in S$ such that $w_i=v_j$ if $i\in I_j$. We say that this set $B$ is a block set with \textit{block size} or simply \textit{degree} $d$.

We are now ready to state the block sets conjecture.
\begin{conjecture}[\cite{LRW}]
Let $m$ and $k$ be positive integers and let $T$ be a template over $[m]$. Then there exist positive integers $n$ and $d$ such that whenever $[m]^n$ is $k$-coloured there exist a monochromatic block set of degree $d$ with template $T$.
\label{mainconj}
\end{conjecture}

Now, for alphabet of size 2, it is easy to
see that this is true for any template. Indeed, this follows
from Ramsey's theorem, and moreover one may take the blocks to
have size 1. To see this, suppose our template has $p$ 1's and $q$ 2's. Given a $k$-colouring of  all the
words in $[2]^n$, let us consider only those words with exactly $q$ 2's. This corresponds to a $k$-colouring of $[n]^{(q)}$, the family of all the $q$-sets in $[n]$. So, by Ramsey's theorem, there is a set $S$ of size $p+q$ all of whose $q$-set
have the same colour, and this corresponds to a block set with the given template -- the $p+q$ blocks are
each of size 1 and are precisely the points of $S$.

In \cite{LRW}, Conjecture \ref{mainconj} is proved for the first
non-trivial template, namely $123$, and more generally for templates
of the form $1\underbrace{22222}_{p}\underbrace{33333}_{q}$ for any positive integers $p$ and $q$. But the conjecture is not known for all templates on alphabet of size 3: the case $112233$ is a
fascinating open problem. We mention in passing that the conjecture is true for the template $1234$ (this follows from 
K\v{r}\'i\v{z}'s result mentioned above), but for $12345$ it is unknown. 

What about the block sizes? It is known (see \cite{LRW}) that,
for the template $123$, block size 1 does not suffice. The proof in \cite{LRW} of the block sets conjecture for template $123$ does not even
have any fixed bound on the block size: the block sizes used
increase as the number of colours $k$ increases. There is also a proof for
template $123$ that may be read out of K\v{r}\'i\v{z}'s result \cite{K} (via the equivalent formulation of the
block sets conjecture mentioned above), and this gives a block
size that is fixed (as the number of colours vary) but is
extremely large (a long tower of integers). In the other
direction, no examples were known of templates for which block
size even 2 does not suffice: the only negative statement was the
fact that block size 1 does not work for $123$ (or any other
template using all symbols $1,2,3$).

One of our aims in this paper is to show that in fact there is no 
upper bound on the block sizes needed. We show that for
every $d$ there exists a template $T$ over $[3]$, and a number of colours $k$, such that whenever we $k$-colour $[3]^n$ (for any $n$) there is no monochromatic block set with template $T$ and degree at most $d$. We stress that the templates we use in this are
all ones for which the block sets conjecture \textit{is} known to be true. 

Our other aim is to show that, for the first non-trivial template,
$123$, in fact blocks of size 2 suffice. This is best possible, in view of the remarks above. Note that this is for any number of
colours. This perhaps might suggest that, for any template, the block size $d$ may be chosen independently of the number of
colours $k$.

The plan of the paper is as follows. In Section 2 we show that block sizes are not bounded, and in Section 3 we prove that
block size 2 suffices for the template $12$3. In Section 4 we
give some open problems. These include a tantalising problem about
Ramsey sets in the $l_1$ norm, which seems to be closely linked 
to some ideas about proving the block sets conjecture.

\section{Unbounded block sizes}

In this section we show that there is no universal bound on the block sizes needed in the block
sets conjecture. In other words, for each $d$ we give a template $T$ and a number of colours $k$ for
which the block size cannot be at most $d$. We mention again that in all cases the templates we consider
\textit{will} be ones that we know do satisfy the block sets conjecture (all are cases that are
proved in \cite{LRW}). In fact, they will all be templates on alphabet $[3]$.

The rough idea of our proof is as follows. We will encode (via some parts of the colouring) information
not only about the `profile' of a word $x$, meaning how many of each symbol occur in $x$, but also about how many times we have (say) a 3 before a 1 
(not necessarily adjacent, but respecting the linear order of the coordinates). In fact, we also
ask about the actual position (modulo some large fixed number) of the 1 in such a pair, so that we
have several separate `contributions'. We make this precise below.

\begin{theorem}
\label{unbounded}
Let $d$ be a fixed positive integer. Then there exists a template $T$ on alphabet~$[3]$, and a colouring of all the finite length words on alphabet $[3]$, such that, for any positive integer $n$, $[3]^n$ contains no monochromatic copy of $T$ whose blocks all have size at most $d$. 
\end{theorem}
\begin{proof}Let $T$ be $1\underbrace{2\dots2}_{d} \underbrace{3\dots3}_{d^3}$. In other words, $T$ consists of a single 1, $d$ copies of 2, and $d^3$ copies of 3. The size of $T$ is $s=1+d(d^2+1)$.
\par For the colouring, we are going to assign (as a colour) to every finite word on $[3]$ a vector of length $d^2+1$ with coordinates in $\mathbb{Z}_{d+1}$ -- in other words, a point in $\mathbb{Z}_{m}^{l}$, where $m=d+1$ and $l=d^2+1$. The additive group notation comes from the fact that a colour will be the sum of certain `contributions'.
\par To start off, let $x$ be a finite word on $[3]$. For each $i$ with $x_i=1$, we define $a_i$ to be the number of 1's and 2's that occur before $i$ when counted from the left, taken modulo $d^2+1$. In symbols, \[a_i=|\{j<i:x_j=1\text{ or }x_j=2\}|\mod d^2+1.\]
\par We define the \textit{contribution} of $i$ to be the basis vector $e_{a_i}$ in $\mathbb{Z}_m^l$. If $x_i\neq1$, then the contribution of $i$ is 0. Finally, we define the colour $c$ of $x$ to be the sum of all of these contributions: in other words, $\sum_{\{i:x_i=1\}}e_{a_i}$. 
\par We will now show that for this colouring there is no monochromatic copy of $T$ with block sizes at most $d$.
\par Suppose for a contradiction that there exists a positive integer $n$, and blocks $B_1,B_2,B_3,\dots,B_s$ all of size at most $d-1$, and $y\in [3]^{[n]\setminus \bigcup_i B_i}$, that yield a monochromatic copy of $T$.
\par By relabelling if necessary, we may assume that, for $i<j$, the first coordinate of $B_i$ is less than the first coordinate of $B_j$. For each $i$, let $b_i$ be the number of 1's plus the number of 2's that occur in $y$ (i.e. not counting any points in the blocks) before the first coordinate of $B_i$, counted modulo $d^2+1$. Since we have $s=1+d(d^2+1)$ blocks and $d^2+1$ possible values for the $b_i$, by the 
pigeonhole principle we must have $d+1$ blocks for which the $b_i$ have the same value, say $b$.
\par Let $B_{i_1},B_{i_2},\dots,B_{i_{d+1}}$ be the corresponding blocks, where $i_1<i_2<\dots<i_{d+1}$. For clarity, let $C_j=B_{i_j}$ for all $1\leq j\leq d+1$. Since we have exactly $d+1$ such blocks, which is the number of 1's and 2's in our chosen template $T$, from now on we will only substitute the single 1 and the $d$ 2's in all the $C_i$, thus forcing all the other blocks to be 3's. 
\par This gives us $d+1$ words $w_1,w_2,\dots,w_{d+1}$, given by which $C_i$ we insert the 1 into. We will derive a contradiction by showing that these $d+1$ words do not, in fact, all have the same colour.

\par We first observe that the colour of one of these words is the sum of the contributions of 
coordinates $\bigcup_i C_i$, together with the contributions from all other 1's, which cannot lie in any $B_i$ (so they are `inactive' or fixed coordinates).
\par Note that the sum of contributions of the `inactive' coordinates is constant since during our substitutions we never replace a 3 with a 1 or 2, or vice versa. Therefore, since we are assuming that the sum of all contributions is constant, the sum of the contributions from the coordinates of $\bigcup_i C_i$ has to be constant too, say $v\in \mathbb{Z}_m^l$.

\par For each $1\leq i\leq d+1$ we denote by $p_i$ the number of 1's and 2's before the first coordinate of $C_i$ in any of the $w_j$ (since the $w_j$ only differ by some 1's changing to 2's and vice versa this is independent of which $w_j$ we choose). By construction, the $C_i$'s have the same number, say $b$, of 1's and 2's mod $d^2+1$ before them, when looking at $y$ only (so not considering any of the blocks). This means that before $C_1$ we have $p_1 \equiv b\mod{d^2+1}$ 1's and 2's.

\par Next, $p_i$ is $b+\lambda_i\mod d^2+1$, where $\lambda_i$ is the number of points of $\bigcup_j C_j$ that occur before the first point of $C_i$.
\par Since all blocks have size at most $d-1$, we get that $\lambda_i\leq (i-1)d\leq d^2$, and so all the $p_i$ are distinct $\mod d^2+1$. Moreover, since the first coordinate of $C_i$ occurs before the first coordinate of $C_{i+1}$ we see that $\lambda_{i+1}\geq \lambda_i+1$. 
\par Finally, observe that that putting 1's in the block $C_i$ contributes to the $p_i\mod{d^2+1}$ coordinate of the colouring. Therefore, by moving the 1 from $C_1$ to $C_i$ adds the contribution $e_{p_i}$ and subtracts the contribution $e_{p_1}$. However, the colour $v$ has to be the same, and thus $e_{p_i}$ has to be the contribution of one of the 1's that were previously in $C_1$. (This is because the other
contributions from this block $C_i$ cannot cancel out the $e_{p_i}$, as $|C_i|\le d$.) Since this is true for any $2\leq i\leq d+1$ and all $p_i$ are distinct $\mod d^2+1$, $C_1$ must produce at least $d+1$ contributions, thus $|C_1|\ge d+1$, a contradiction.
\end{proof}

We remark that the above proof is unchanged if one wishes for a more general template $T=1\underbrace{22\dots2}_{p}\underbrace{33\dots3}_{q}$. The length of $T$ is $s=1+p+q$. The colouring is the same -- associate to every word over $[3]$ a point in $\mathbb{Z}_m^l$ as above. Provided $s$ is at least $pl+1$, we still get $p+1$ blocks with the same value for $b_i$. By taking $l>pd$, we ensure that the $\lambda_i$ are still disjoint, and since $m>d$ we still produce at least $d+1$ contributions. Hence, if $s\geq pl+1$, $pd<l$ and $m>d$, the above proof goes through identically.  So what we obtain is that
a template with one 1, $p$ 2's and $q$ 3's does need block 
sizes greater than  $d$ whenever $q\geq p^2d$. 

If one considers the natural templates $123\dots m$, it is
easy to check that if the block sets conjecture holds for
one such template $123\dots m$ then it also holds for all
templates of length $m$. This is informally because identifying
symbols `maps block sets to block sets'. Hence it also follows from Theorem \ref{unbounded} that for any $d$
there is a template of the form $123\dots m$ for which one
cannot guarantee that one can take blocks of size at most $d$.

We mention that, for Euclidean Ramsey sets, the result above has the following consequence. For any
$d$, there exists a transitive set $X$ that is Ramsey, and a number of colours $k$, such that whenever
the set $(1/\alpha) X^n$ is a $k$-Ramsey set for $X$ we must have $\alpha \geq d$. This may be read
out of the equivalent form of the block sets conjecture in terms of product sets mentioned above.

\section{The template $123$}

As we have seen, the block sets conjecture is essentially trivial for templates over the alphabet of size 2,
and even with blocks of size 1. So we now turn our attention to the first non-trivial template, namely $123$.
Here it was known that the blocks cannot be taken to be of size 1, but the block sizes that were known
to work were huge. Indeed, in \cite{K} the block size is a tower, coming from an iterated application
of some Ramsey theorems. And in \cite{LRW} there is no iteration, but the proof uses van der Waerden's theorem
for $k$ colours, so that the block size is not even constant but grows with the number of colours.

Our aim is to show that block size 2 suffices. This is best possible, in view of the remarks above.
In fact, we show that a fixed pattern suffices: one can always find a block set whose \textit{pattern}
(the way that the set of $2\cdot3$ coordinates is partitioned into the 3 blocks) is fixed, namely
$ABCCBA$. (Here $A,B,C$ represent the relative positions of the 2-sets $I_1,I_2,I_3$ in the fixed linear
order on $[n]$.) The key idea is in fact this palindromic aspect (the use of reversals). In contrast, the
earlier proofs used patterns like $ABCABCABC\dots ABC$, or iterated versions of that, and as we will see
(see the discussion after the proof) there are reasons why this makes the proof much harder.

\begin{theorem} Let $T$ be the template $123$, and let $k$ be a positive integer. Then there exists $n$ such that whenever $[3]^n$ is $k$-coloured there exists a monochromatic copy of $T$ with pattern $ABCCBA$.
\label{123}
\end{theorem} 
\begin{proof} 
For simplicity and readability, we will just think of $n$ as being `very large': the actual value of $n$ will be fixed during the proof.
\par Let $\theta:[3]^n\rightarrow [k]$ be a fixed $k$-colouring of $[3]^n$. We must show that we can find $n-6$ fixed coordinates and 6 active ones, labelled $ABCCBA$ in this order, such that whenever we substitute $1,2,3$ (each corresponding to exactly one of $A$, $B$ and $C$), all 6 resulting words have the same colour.
\par Our first step is to ensure informally that, by passing to a subset of the coordinates, `the positions of the
3's do not matter'. This is a fairly standard kind of step to take in such a Ramsey argument.

Let $\mathcal A$ be the set of words of length $n$ (on alphabet $[3])$ that contain exactly $2k+2$ 2's and no 1's. For example, if $k=2$ and $n=10$, such a word could be $2322333222$.
\par Let also $\chi$ be the set of words of length $2k+2$ on alphabet $[2]$ that have exactly $k+1$ 1's and $k+1$ 2's. For example, if $k=2$, then words such as $111222$, $121212$, $1221211$ are all in $\chi$. Clearly, the size of $\chi$ is $s=\binom{2k+2}{k+1}$, and so let $\chi =\{w_1, w_2, \dots, w_s\}$.
\par Given a word $x\in\mathcal A$ and a word $w\in\chi$, we can generate a unique word in $[3]^n$ by replacing the 2's in $x$ with the word $w$ -- in other words, at the positions of the 2's in $x$ we insert, in order, the letters of $w$. Call this word $f(x,w)\in[3]^n$. As an  
example, if $k=2$, $n=10$, $x=2322333222$ and $w=121212$, then $f(x,w)=1321333212$.
\par It is clear that $f$ is a bijection between $\mathcal A\times\chi$ and the set of all the words of length $n$ with exactly $k+1$ 1's and $k+1$ 2's. This allows us to induce a colouring $\mathcal A$ with $k^s$ colours as follows: $\Theta:\mathcal A\rightarrow [k]^s$ is given by $\Theta(x)=(\theta(f(x, w_1)),\theta(f(x,w_2)),\dots,\theta(f(x,w_s))).$
\par Next, a word $w$ in $\mathcal A$ can be seen as a subset of $[n]$ of size $2k+2$, by selecting the coordinates at which the 2's are. So we may also view $\Theta$ as a $k^s$-colouring of $[n]^{(2k+2)}$.

\par By Ramsey's theorem, there exists $n$ such that whenever $[n]^{(2k+2)}$ is $k^s$-coloured there exists a set $S$ of size $2k+4$ such that $S^{(2k+2)}$ is monochromatic. (This is the actual version of `$n$ is
very large' that we need.) 
\par Applying this to our colouring $\Theta$, and assuming for clarity that $S=[2k+4]$, we get that no matter how we insert $2k+2$ 2's and 2 3's in the first $2k+4$ coordinates, we get the same colour under $\Theta$. This means that for any $w\in\chi$, if $x_1, x_2\in\mathcal A$ have their 2's in the first $2k+4$ coordinates, then $\theta(f(x_1,w))=\theta(f(x_2,w))$. In other words, no matter how we insert the 2 3's in the first $2k+4$ coordinates, the colour does not change, as long as the $12$ word is the same, namely $w$. In other words, informally, `the positions of the 3's do not matter'.
\par Now  consider the following $k+1$ words in $\chi$:\\
\begin{align*}
z_1&=211212\dots12\\z_2&=122112\dots12\\&\qquad\vdots\\z_{k+1}&=121212\dots21   \end{align*}
In words, $z_i$ has 1's at precisely positions $1,3,\dots, 2i-3, 2i,2i+1\dots2k+1$ -- or, simply put, it is a concatenation of $k+1$ `12' blocks, with the $i^{\text{th}}$ block flipped to a `21'.
\par Let $x\in\mathcal A$ be the word that starts with $2k+2$ consecutive 2's, the rest being all 3's. Consider the words $f(x,z_i)$ for $1\leq i\leq k+1$. Since we have only $k$ colours, by the pigeonhole principle there exists $i<j$ such that the colour of $f(x,z_i)$ is the same as the colour of $f(x,z_j)$. However, by the above choice of $S$, this colour does not change no matter how we reorder the first $2k+4$ coordinates of $x$.
\par To conclude the proof, we take the blocks to be $\{2i-1,2j+2\}$, $\{2i, 2j+1\}$ and $\{2i+1, 2j\}$, and the fixed coordinates are the $k-1$ remaining `12' blocks in the first $2k+4$ coordinates and all 3's in rest. This gives a monochromatic copy of $123$ with pattern $ABCCBA$, as required.

\par To help the reader visualise the above construction, let us look at a concrete example. Suppose $k=3$ and $f(x,z_1)$ has the same colour as $f(x, z_2)$. In other words the following receive the same colour, say $a$:

\begin{align*} f(x,z_1)=\quad&\,2112121233\dots\\f(x, z_2)=\quad&\underbrace{1221121233}_{\mathclap{2k+4=10\text{ coordinates}}}\dots\end{align*}

\par By moving the two 3's around in the first word, we see that the following all have colour $a$:
\begin{center}$\begin{array}{c}
3211231212\dots\\
2311321212\dots\\
2133121212\dots
\end{array}$
\end{center}
\par Doing the same thing for the second word, we see that the following also have colour $a$:
\begin{center}$\begin{array}{c}
3122131212\dots\\
1322311212\dots\\
1233211212\dots
\end{array}$
\end{center}
And so the first 6 coordinates, under the pattern $ABCCBA$, give a monochromatic copy of $123$.
\end{proof}

We remark that the above proof goes through in the exact
same way more generally, for templates of the form
$12\underbrace{33333}_{q}$, again yielding block size 2.
Of course, by Theorem \ref{unbounded} it cannot go through
in the same way for templates of the form
$1\underbrace{22222}_{p}\underbrace{33333}_{q}$. However, it
is possible that for each such template there could be a fixed
block size that one may always take for that template
(for any number of colours). 

An alternative viewpoint for the second half of the above argument (after the positions of the 3's do not matter) is to identify a word with a vector over $\mathbb{Z}$ as follows. We fix a `reference' word $w$ with $t$ 1's. Then, for any other word $v$ with $t$ 1's, we identify $v$ with the point in $\mathbb{Z}^t$ that represents how the 1's have moved between $w$ and $v$. In other words, if $w$ has 1's in positions $a_1,\dots,a_t$ then if $v$ has 1's in positions $a_1+b_1,\dots,a_t+b_t$
then it corresponds to the vector $(b_1,\dots,b_t)$. So we may view a colouring of the words (with $t$ 1's)
as a colouring of $\mathbb {Z}^t$. (Strictly, one would need to bound the coordinates of the points in
$\mathbb {Z}^t$ that we are considering, as one would not want the $t$ 1's to overlap, for example -- we
ignore this point in the interests of clarity.)

In this language, the final step of the argument above is just the assertion that, whenever we
$k$-colour $\mathbb {Z}^t$ (for fixed $k$ and for sufficiently large $t$) there exist two vectors
of the same colour that differ by a vector of the form $e_i-e_j$.

It is interesting to compare this with the approach taken in \cite{LRW}, where no reversals appear.
There the idea is as follows. If we
used pattern $ABCABCABC\dots ABC$ (say with each block set being of size $d$) then we would be asking
for two points $x$ and $x+v$ of the same colour, in a $k$-colouring of $\mathbb{Z}^n$ with $n$ large,
such that $v$ has $d$ 1's and all other coordinates zero. This is of course impossible, as we may
just 2-colour by asking whether or not the coordinate sum, taken modulo $2d$, lies in $[0,d-1]$. 
So one has to
consider also patterns like $AABBCCAABBCC\dots AABBCC$, where one would be asking for $v$ having
$d$ 2's and all other coordinates zero, and so on. It is this (which of course works for enough
candidate patterns, thanks to van der Waerden's theorem) that means that one of these fixed
patterns by itself cannot work.

To end this section, let us record that Theorem \ref{123} has the following direct consequence for Ramsey sets.
If $X$ is a regular hexagon, then for any $k$ the set $X^n$, shrunk only by a factor $\sqrt{2}$, is
$k$-Ramsey for $X$. And the same would hold for any plane hexagon on which $S_3$ acts transitively -- 
so any hexagon whose internal angles are all $60^{\circ}$ and whose sides alternate in length as
$a,b,a,b,a,b$. This follows from the equivalent formulation of the block sets conjecture in terms of
product sets -- see \cite{LRW}. 

\section{Discussion and open problems}

One of the most interesting questions raised by the previous section is whether or not this
phenomenon, of the block size being constant as the number of colours grows, holds for any template.

\begin{conjecture}
Let $m$ be a positive integer and let $T$ be a template over $[m]$. Then there exists $d$ such that
for any $k$ there is an $n$ such that whenever $[m]^n$ is $k$-coloured there exist a monochromatic block set of degree $d$ with template $T$. 
\label{boundedq}
\end{conjecture}

Note that this would actually imply that one can ask for the block size \textit{and} the pattern of the
blocks to be fixed in advance. This is because the number of patterns for a given block size is finite,
so if no pattern could be guaranteed (for block size $d$) then a product colouring would yield no
block sets of size $d$ at all.

As mentioned above, the proof of Theorem \ref{123} goes through in the 
same way for the template $12\underbrace{33333}_{q}$, so that
Conjecture \ref{boundedq} holds for these templates. The
templates $1\underbrace{22222}_{p}\underbrace{33333}_{q}$ are the first examples of templates for which we know that the block sets conjecture does hold and yet we do not know if
Conjecture \ref{boundedq} holds.

Of course, the most important problem on the block sets conjecture is
the full conjecture itself. The simplest unknown case, in the sense of being on
a small alphabet, is the template
$T=112233$. This seems to be the key `next step', with the 
repetition of all the symbols being the crucial obstacle.

Digressing for a moment, we mention again that, for larger alphabets, the
template $1234$ is known to satisfy the block sets conjecture, which may be read out of the result
of K\v{r}\'i\v{z} \cite{K}. Indeed, K\v{r}\'i\v{z} showed that any set that has a group of symmetries that is soluble and acts transitively is Ramsey, and this, translated into block sets language, gives that
the template $1234$ satisfies the block sets conjecture because the group
$S_4$ is soluble. However, the template $12345$ is open -- the fact that the group
$S_5$ is not soluble means that K\v{r}\'i\v{z}'s result does not apply.  (In fact, K\v{r}\'i\v{z}'s proof of his result actually always gives that the block sizes may be taken to be fixed -- so that for example the template $1234$ does
satisfy Conjecture \ref{boundedq}. This  
is perhaps more evidence towards Conjecture \ref{boundedq}.)

Returning to the template $112233$, and continuing the line of thought as given in the 
explanation after the proof 
of Theorem \ref{123}, let us start by considering patterns, say with all blocks of size $d$, of the form
$ABCDEFABCDEF \dots$ -- this is not palindromic, but let us ignore that for the moment. When 
one `moves' the 1's from say $A$ and $B$ to $B$ and $C$ (with 2's now moving to $A$), and then to
$C$ and $D$ (with 2's now in places $A$ and $B$), 
then the positions of the 1's change by 1 in the first case (as explained above), and by 2 in
the second case. If we used some palindromic version, so with repeats of $FEDCBA$ as well, then that
would merely change some of these $+1$s and $+2$ into $-1$s and $-2$s.

So we would be asking the following question. Given $k$, must there exist $d$ such that whenever we $k$-colour $\mathbb{Z}^n$, for $n$ large,  
we always find monochromatic $x,x+v,x+2v$ where $v$ has $d$ 1's and $d$ $-1$s and the rest of the
coordinates are zero? Sadly this is not true, for example because this $v$ has fixed $2$-norm, so that
we would be asking for a certain non-spherical set (an arithmetic progression of length 3) to be
Ramsey. 

Now, if instead the pattern were a concatenation of words of the form $AABBCCDDEEFF$ (perhaps made
palindromic), then these moves would be by $+2$ and $+4$ respectively (with $-2$ and $-4$ when we use the
reversed word). And similarly for repeats of $AAABBBCCCDDDEEEFFF$ and so on. Hence, if we asserted that our 
pattern consisted of a fixed number of $ABCDEF$ words, a fixed number of $AABBCCDDEEFF$ words, and so
on, then the vector representing the positions of the 1's in the $12$-word would take values like
$x,x+v,x+2v$, where $v$ would have a fixed form consisting of $2t$ values, of which a given number would
be $\pm 1$s, a given number $\pm 2$s, and so on. It turns out, however, that some similar behaviour
is exhibited by those patterns where the $l_1$ norm of this $v$ is the same -- for example, the 1-norm seems also to come in when we move the 1's to other places in the template. So it actually makes
sense to ask for $v$ to have not just fixed support size (of $2d$) but also fixed $1$-norm. This
means that if we have $\lambda_1$ parts of the pattern of the form $ABCDEF$, $\lambda_2$ parts of the
form $AABBCCDDEEFF$, and so on, the we would be asking for a $v$ of support size $2d$ whose support is
made up of $\lambda_1$ values of $1$ and the same of $-1$, and $\lambda_2$ values of $2$ and the same of
$-2$, and so one, where the sum of all the $i\cdot\lambda_i$ equals $d$.

Now, of course this is a very special form for a vector of 1-norm $2d$. Nevertheless, we feel
that asking the question for a general vector is a sensible one, and perhaps is relevant for
making further progress. This is perhaps rather speculative, but it also seems just too nice not to be asked.

\begin{question}
Given $k$, do there exist $d$ and $n$ such that whenever $\mathbb{Z}^n$ is $k$-coloured there are
$x$ and $v$ with $\|v\|_1 = d$ such that all of $x-v,x,x+v$ have the same colour?
\label{l1q}
\end{question}

One important point is that there is a linear, not just a metric, constraint here: we are not just
asking for 3 points with pairwise $l_1$ distances $d,d,2d$. In fact, the latter is quite easy to
achieve by a direct argument. Of course, in $l_2$ these pairwise distances would imply collinearity, but
not in $l_1$. 

Now, for the 2-norm in place of the 1-norm Question \ref{l1q} would be false, as the original argument of \cite{EGMRSS} shows - since an arithmetic progression of length 3 is not spherical. In the other direction, for the $l_\infty$ norm the
result is true. As noted by Kupavskii and Sagdeev \cite{KS}, this may be read out of the
Hales-Jewett theorem instantly. Indeed, any monochromatic line in
$[3]^n$ corresponds to an arithmetic progression with common difference having $l_\infty$ norm 1. (See that paper, and also Frankl, Kupavskii and Sagdeev \cite{FKS}, for several strong bounds on the actual dimensions needed.) There is also some interesting related work on all norms $l_p$ except
for $p=1$ -- see Cook, Magyar and Pramanik\cite{CMP} --  but the problem above remains
tantalisingly open. 

The above should just be the start, and there is a vast generalisation that one could also ask for. If the answer to Question \ref{l1q} is
in the affirmative, then actually one would want longer arithmetic progressions, and indeed one would 
want larger parts of 
$l_1$ balls -- where the set $x-v,x,x+v$ in Question \ref{l1q} is viewed as just one very small part of an
$l_1$ ball of radius $d$ (centred at $x$). More usefully, one could also view this set (informally) as 
being the one-dimensional 
$l_1$ ball of radius 1 `generated' by the single vector $v$ (again centred at $x$). Indeed, some version of (some part of) this is needed even for the
template $112233$, since in the above we only considered the 1's being in positions $A$ and $B$ or $B$ and $C$ or $C$ and $D$ -- it turns out that looking at other places does correspond to taking other vectors
than $x,x+v,x+2v$ and indeed these vectors do turn out to all lie on a constant radius $l_1$ sphere about some point. So again answering the question below may well prove relevant to progress on the
block sets conjecture for the template $112233$ and beyond.

In general, we define the \textit{$l_1$ ball of radius $r$ generated by $u_1,\dots,u_t$}, where
$u_1,\dots u_t$ are disjointly-supported points in $\mathbb{Z}^n$, to consist of all sums $\sum \lambda_i u_i$, where
the $\lambda_i$ are integers satisfying 
$\sum |\lambda_i| \leq r$. 

\begin{question}
Given $r$ and $t$ and $k$, do there exist $d$ and $n$ such that whenever $\mathbb{Z}^n$ is $k$-coloured there 
are disjointly-supported vectors $u_1,\dots,u_t$, each of 1-norm $d$, such that some translate of the 
$l_1$ ball of radius $r$ generated by $u_1,\dots,u_t$ is monochromatic?
\label{generalq}
\end{question}

As explained above, the case $r=t=1$ of this is precisely Question \ref{l1q}. In fact, we suspect that the hard part is Question \ref{l1q}: if it is true then
perhaps Question \ref{generalq} ought to be true as well. Note that, because any finite subset of 
$\mathbb{Z}^n$ is contained in an $l_1$ ball, 
Question \ref{generalq} has the attractive Ramsey flavour that the objects we are colouring and the
objects that we seek a monochromatic copy of both have the same form. 

Finally, we mention that, beyond the block sets conjecture itself, a key
question is which (if either) of the two rival conjectures about Euclidean
Ramsey sets is correct. There are explicit sets known that are spherical but do not embed into any transitive set: one example is a cyclic kite,
namely the four points $(\pm 1, 0)$ and $(a, \pm \sqrt{1-a^2})$, with $a$ being transcendental (see \cite{LRW2}). Determining whether or not this quadrilateral is Ramsey would immediately
rule out one of the two conjectures.

\bibliographystyle{amsplain}
\bibliography{document}
\Addresses
\end{document}